\documentclass[12pt]{amsart}
\usepackage{pstricks}
\usepackage{amsfonts}
\usepackage{amsthm}
\usepackage{latexsym,amsmath}
\usepackage{graphicx}
\usepackage{stmaryrd}

\newtheorem{theorem}{Theorem}[section]
\theoremstyle{plain}
\newtheorem{Definition}[theorem]{Definition}

\newenvironment{theorem*}[1]{\medskip
                            \noindent
                            {\bf Theorem #1. }\ %
                            \begingroup \sl}
                            {\endgroup\medskip}

\newtheorem{Proposition}[theorem]{Proposition}

\newtheorem{lemma}[theorem]{Lemma}

\begin{document}

\title
{A proof of P$\neq$NP}

\author[$\mathrm{M^{\lowercase{c}}Callum}$ ]{\textbf{Rupert} $\mathbf{M^{\lowercase{c}}Callum}$ }

\begin{abstract}

\noindent This work will show that $P \neq NP$ is provable in $\textsf{PA}$ and indeed even in $\textsf{SEFA}$, an extension of $\textsf{EFA}$ which proves the totality of super-exponentiation. The opening section gives acknowledgements and the second section makes some introductory remarks. A discussion of which arithmetic theories are used is given in the third section, and an outline of the overall proof-strategy is given in the fourth section. The fifth section gives a complete presentation of every step in the argument for the purpose of proving the result in $\textsf{PA}$.

\bigskip

\noindent A lemma stated in the fourth section and proved in the fifth section outlines a construction of an arithmetically definable sequence $\{\phi_{n}:n \in \omega\}$ of $\Pi^{0}_{2}$-sentences which are each true in the standard model, which satisfy a number of properties, and of which we can say, speaking roughly, that the growth in logical strength is ``as fast as possible", manifested in the fact that the total general recursive functions whose totality is asserted by the $\Pi^{0}_{2}$-sentences in the sequence are cofinal growth-rate-wise in the set of all total general recursive functions.

\bigskip

\noindent We then develop an argument, which makes use of a sequence of sentences constructed from the preceding sequence by means of a certain somewhat involved construction, together with an application of the diagonal lemma. These sentences in this second sequence are generalisations, broadly speaking, of Hugh Woodin's ``Tower of Hanoi" construction as outlined in his essay in Chapter 18 of \cite{Woodin1998}. The argument establishes the result that it is provable in $\mathsf{PA}$ that $P \neq NP$ (careful examination of this proof seems to show that $\Sigma^0_3$-induction arithmetic is sufficient). The final section shows how to establish the stronger claim that the result $P \neq NP$ is actually provable in $\mathsf{SEFA}$.

\end{abstract}

\maketitle

\newpage

\section{Acknowledgements}

\noindent I am thankful to Henry Towsner for taking some care and patience to critique flawed earlier versions of parts of this argument. I'm also very thankful to Benoit Monin who very kindly closely examined the argument and pointed out important corrections. I'm also thankful to an anonymous reviewer of an earlier draft who identified issues with an earlier attempted proof of Proposition \ref{Step1}.

\section{Introduction}

\noindent In what follows, we will present an argument establishing that $P\neq NP$ is provable in first-order $\mathsf{PA}$, using strategies based on a diagonal-lemma construction as clarified in the abstract. The argument in the most natural form is highly non-constructive and seems to require $\Sigma^0_3$-induction arithmetic, but with some care it appears to be to possible to pull the argument down into $\mathsf{SEFA}$, as will be clarified in the final section. The following section outlines the arithmetical theories that shall be used in the course of the proof, the fourth section presents the overall shape of the proof, and the fifth section presents the detailed proof of the main result, while the final section explains how to get the argument to go through in $\mathsf{SEFA}$.

\section{The arithmetical theories that shall be used}

In this section, we shall survey the various arithmetical theories to which we shall be referring in the exposition of the main argument. At one point in the exposition of the actual argument, we shall need to state a lemma to the effect that a certain decision problem related to the syntax of first-order languages is in $NP$. Since decision problems which are candidates for being in $NP$ must be decision problems for languages in finite alphabets, we shall need to describe how we shall view first-order languages as having finite alphabets.

\bigskip

\noindent We shall do this by assuming that the variables of our first-order languages are of the form $x, x', x'', \ldots$ where $x$ and $^{\prime}$ are two symbols of the language. We shall also assume that all first-order languages have finite signature except possibly for having countably infinitely many constant symbols, but in the case where the number of constant symbols is infinite, we shall allow for constant symbols being represented by a string of symbols of our alphabet of length greater than one, and we shall always require that the alphabet is finite. The alphabet shall consist of $x$ and $^{\prime}$, and two symbols $^{*}$ and $^{\circ}$ whose usage shall be described below, and finitely many non-logical symbols and all the other standard symbols of a first-order language.

\bigskip

\noindent There is one particular example of a theory in such a first-order language to which we shall need to refer regularly throughout the argument. This theory is a modified version of $\textsf{PRA}$. We shall describe it.

\bigskip

\noindent Before proceeding to describe this theory, let us remind ourselves of the devices for quotation introduced by George Boolos in \cite{Boolos1995}. We assume that we are working in a language in some finite alphabet which includes the symbols $^{*}$ and $^{\circ}$.

\begin{Definition} \label{qmark} A $q$-mark is an expression $\{j\}^{\circ}$ where $j \geq 0$ and $\{j\}$ is the expression consisting of $j$ consecutive occurrences of $^{*}$. \end{Definition}

\begin{Definition} \label{quotation} If $\alpha$ and $\beta$ are two expressions, then we write $\alpha * \beta$ for the concatenation of $\alpha$ and $\beta$. For any expression $\alpha$, we define $m(\alpha)$, the $q$-mark of $\alpha$, to be the shortest $q$-mark that is not a subexpression of $\alpha$. The quotation of $\alpha$, denoted by $r(\alpha)$, is defined to be $m(\alpha) * \alpha * m(\alpha)$. \end{Definition}

\noindent Now let us describe our modified version of $\textsf{PRA}$. The language shall be an extension (by constant symbols) of the first-order language of arithmetic, which we shall assume to have signature $(0,S,+,*,<)$. We are making the convention, of course, that the alphabet of the language is finite, and that all variables are given by means of two fixed symbols in the manner outlined previously. Let us denote the first-order language of arithmetic as just described (not extended by constant symbols as yet) by $L_0$. We shall begin with a theory denoted by $T_0$ in the language $L_0$ just described, which is a version of $\textsf{PRA}$ in that language, which we shall assume to have a polynomial-time-computable axiom set.

\bigskip

\noindent Now let us describe an extended theory $T_1$ in an extended language $L_1$. Presently we shall be describing a sequence of languages $\langle L_n :n \in \omega \rangle$ with the property that for each $n \in \omega$ the language $L_{n+1}$ is an extension of the language $L_{n}$ by countably infinitely many constant symbols which may be represented by strings of symbols from the alphabet of length greater than one. We shall assume in what follows that some fixed system of G\"odel numbering for the union of all these languages is chosen. To return to the description of the theory $T_1$, we shall take the language $L_1$ to be the same as that of $L_0$, only with countably infinitely many constant symbols added, where we keep the size of the alphabet finite by allowing constant symbols to have length greater than one. Given any sentence $\phi$ of $L_0$ with length $n$ (containing no $q$-marks), we stipulate there is a constant symbol $c_{\phi}$ of $L_1$ of length $n+2$ which appears as the expression $r(\phi)$, as given in Definition \ref{quotation}, and an axiom $c_{\phi}=n$ is given for the theory $T_1$, where $n$ is a closed term in the language $L(0,S,+,*,<)$ for the G\"odel number of $\phi$, with the length of the closed term $n$ being bounded above by a bound proportional to the logarithm of the G\"odel number which is its referent in the standard model. Clearly by proceeding in this way we can arrange things (if the system of G\"odel numbering is chosen appropriately) so that the axiom set for the theory $T_1$ is once again polynomial-time-computable.

\bigskip

\noindent Then we can construct an extended theory $T_2$ in an extended language $L_2$ in a similar way, except that now given a sentence $\phi$ of length $n$ in the language $L_1$ with at least one $q$-mark of length 1, the corresponding constant symbol will have length $n+4$, since $q$-marks of length 2 will be needed to do the quotation. And so on. In this way we construct an infinite ascending chain of theories $\{ T_n : n \in \omega\}$. The union of all such theories, which we denote by $T_{\infty}$, is a theory in a language with finite alphabet with a polynomial-time-computable axiom set. This theory by slight abuse of notation shall also be denoted by $\textsf{PRA}$, throughout what follows, and the language of all arithmetical theories shall hence-forward be assumed to be the underlying language of this theory.

\bigskip

\noindent It may be asked why we have chosen to set up the theory $\textsf{PRA}$ in this particular way. This is because it is necessary to do so in order to guarantee, in a lemma to be stated in a later section, that the sequence of $\Pi^0_2$-sentences referred to in the statement of that lemma only have polynomial growth in their length.

\begin{Definition} Let $\{\Gamma_r :r \in \omega\}$ be the following increasing sequence of theories. $\Gamma_0$ is defined to be the deductive closure in the predicate calculus for $L(0,S,+,*)$ of the set of all consequences of the axioms of $\mathsf{Q}$ which can be formulated in $L(0,S)$ together with bounded induction for that language. $\Gamma_1$ is defined to be the deductive closure in $L(0,S,+,*)$ of Presburger arithmetic. $\Gamma_2$ is $\mathsf{Q}$ together with induction for bounded formulas. For $\Gamma_3$, extend the language to include a symbol for exponentiation and include the defining recursion equations for exponentiation, and then add bounded induction axioms for that language, and take the set of consequences of this axiom set in $L(0,S,+,*)$. Similarly with $\Gamma_4$, except that we include a symbol for super-exponentiation as well as exponentiation. And with $\Gamma_5$ we have symbols for exponentiation, super-exponentiation, and super-super-exponentiation. And so on. We previously said that all arithmetical theories would be in the language for the modified version of $\mathsf{PRA}$ with infinitely many constant symbols, some of them strings of length greater than one, so we make that modification to the sequence of theories just described so as to conform to that convention. This completes the specification of the hierarchy of theories $\{\Gamma_r :r \in\omega\}$ to which we will refer in what follows. \end{Definition}

\bigskip

\noindent In particular, with $\{\Gamma_r :r \in \omega\}$ as above, we have that the set of provably total general recursive functions of $\Gamma_r$ is equal to the $r$-th set in the Grzegorczyk hierarchy of primitive recursive functions.

\bigskip

Next we define $I\Delta_0+\{\Omega_r:r\in\omega\}$. We define $I\Delta_0$ as $\Gamma_2$ and $\Omega_r$ is defined for each $r \in \omega$ as follows. We define $\omega_0(x)=2x$ and we let $|x|$ be the least $y$ such that $2^y>x$, then we define $\omega_{i+1}(0)=0$, and for $x>0$, $\omega_{i+1}(x)=2^{\omega_i(|x|-1)}$. Then $\Omega_r$ is the assertion of the totality of the function $\omega_{r+1}$. Here we are following unpublished work of Peter Koellner and Hugh Woodin for which a citation may be found in \cite{Paris&Wilkie1987}.

\bigskip

We shall also be referring to the theories $\Sigma^0_n$-induction arithmetic for $n>0$ and $\textsf{PA}$ which have standard definitions.

\bigskip

This completes the specification of the arithmetical theories to which we shall be referring in the course of the argument.

\section{Outline of the argument}

\noindent This section and the following section shall be occupied with presenting a proof which is formalizable in $\mathsf{PA}$. This section shall give the outline of how the proof goes and the fifth section shall fill in the details. We shall need to begin by stating a lemma to the effect that a certain decision problem related to the syntax of first-order languages is in $NP$. It shall indeed become apparent over the course of the exposition of the argument that it is $NP$-complete. We shall make use of the notation introduced in the previous section.
\bigskip

\noindent We shall now state a lemma which shall give an example of a decision problem in $NP$.

\begin{lemma} \label{decisionproblem} Suppose that we are given a finite list of axioms to add to $\textsf{PRA}$ and a positive integer $n$ given in unary notation, and a particular sentence $\phi$, and the decision problem is to determine whether or not there exists a proof in $\textsf{PRA}$ plus the given list of axioms, of length at most $n$, of the sentence $\phi$. This decision problem, denoted in what follows by $\textsf{PROVABILITY}$, is in $NP$. \end{lemma}

\begin{proof} It is sufficient to show that the decision problem for determining
whether a given string of length at most $n$ is a proof in $\textsf{PRA}$ plus a fixed finite list of axioms can be solved in a number of steps which is a polynomial function of $n$. Recall that we are assuming that the axiom set of $\textsf{PRA}$ is polynomial-time computable, and it easily follows that the axiom set of $\textsf{PRA}$ with the finite list of axioms adjoined is also polynomial-time-computable. We may also assume that the logic itself is polynomial-time-computably axiomatized relative to the rule modus ponens. It is clear, then, that the decision problem for determining whether a given string is a proof in the theory $\textsf{PRA}$ plus the finite list of axioms is polynomial-time-computable, and this yields the desired result. \end{proof}

\noindent Having stated our initial lemma, let us now proceed to sketch how the rest of the proof goes, with the detailed steps to be given in the fifth section.

\bigskip

\noindent We state another lemma, to be proved in the fifth section.

\begin{lemma} \label{Lemma} Assume the axioms of $\mathsf{PA}$. On those assumptions, it is provable that there exists an arithmetically definable sequence of $\Pi^{0}_{2}$-sentences $\{\phi_{n}:n \in \omega\}$ in the language of $\textsf{PRA}$ as defined above (we allow that they may be finite conjunctions of sentences of complexity either $\Pi^0_1$ or $\Pi^0_2$, and thus provably equivalent in $\textsf{PRA}$ to some $\Pi^0_2$-sentence) with the property that

\bigskip

\noindent (1) No $q$-marks of length greater than $n$ occur in $\phi_{n}$ (and we start with $n=0$). \newline
\noindent (2) There exists a polynomial function with non-negative integer coefficients such that the length in symbols of $\phi_{n}$ is always bounded above by the value of this polynomial function at $n$, and is also strictly greater than the value of this polynomial function at $n-1$ if $n>0$; \newline
(3) $\mathsf{PRA}$+$\{\phi_{n}:n\in\omega\}$ is $\Pi^{0}_{2}$-sound and $\Pi^{0}_{1}$-complete \newline
(4) $\mathsf{PRA}$+$\phi_{n+1}$ always proves 1-consistency of $\mathsf{PRA}$+$\phi_{n}$, in a form where the G\"odel number of $\phi_{n}$ is denoted by a constant symbol with $q$-marks, and the length of the shortest proof as a function of $n$ is bounded above by a polynomial function with non-negative integer coefficients; \newline
(5) The provably total general recursive functions of $\mathsf{PRA}$+$\{\phi_{n}:n \in\omega\}$ have growth rates cofinal in the set of growth rates of all the total general recursive functions. \end{lemma}

\noindent Apart from the proof of this lemma, there are three steps that must be completed to complete the proof in $\textsf{PA}$ that $P \neq NP$. We shall describe them.

\bigskip

\noindent \textbf{Description of Step 1.} Let us suppose, for a contradiction, that $P=NP$ holds. Then there would be some algorithm for deciding $\textsf{SAT}$ in polynomial time. Consequently there would be some Turing machine $T$ and some non-negative integer $m'$ such that the $\Pi^0_1$-sentence, asserting of the Turing machine $T$ that it always correctly computed of a given input string whether or not that string was a member of $\textsf{SAT}$, in a number of steps bounded above by $p(n)$, where $n$ is the length of the input string and $p(n)$ is a polynomial function with non-negative integer coefficients and leading exponent $m'$, would hold true. Since the decision problem $\textsf{PROVABILITY}$ is polynomial-time reducible to $\textsf{SAT}$ by the Cook-Levin Theorem, there would exist some non-negative integer $k$ such that if we define $m:=km'$, then there would be some Turing machine $T'$, such that the $\Pi^0_1$-sentence, asserting of the Turing machine $T'$ that it always correctly computed of a given input string whether or not that string was a member of $\textsf{PROVABILITY}$, in a number of steps bounded above by $p'(n)$, where $n$ is the length of the input string and $p'(n)$ is a polynomial function with non-negative integer coefficients and leading exponent $m$, would hold true. Let us call such a $\Pi^0_1$-sentence, a $\Pi^0_1$-sentence witnessing the polynomial-time computability of $\textsf{PROVABILITY}$ with exponent $m$. Consequently the latter $\Pi^0_1$-sentence would be provable in $\textsf{PRA}+\phi_{n}$ for all $n \geq N$ for some fixed $N$, where $\{\phi_{n}:n \in \omega\}$ is the sequence of $\Pi^0_2$-sentences of Lemma \ref{Lemma}.

\bigskip

\noindent Now suppose that $m$ is a fixed integer such that $m \geq 2$, and suppose that we are given a quadruple $\langle \phi, \phi', P, P' \rangle$, where $\phi$ and $\phi'$ are $\Pi^0_2$-sentences (again, they may possibly be finite conjunctions of $\Pi^0_1$ and $\Pi^0_2$ sentences), and $P$ is a proof from $\textsf{PRA}+\phi'$ of the 1-consistency of $\textsf{PRA}+\phi$, and $P'$ is a proof from $\textsf{PRA}+\phi$ of a $\Pi^0_1$-sentence witnessing the polynomial-time computability of $\textsf{PROVABILITY}$ with exponent $m$. We are not assuming here that $\phi$ and $\phi'$ are actually true in the standard model, but special cases of this, which we shall consider later, may be examples where $\phi$ and $\phi'$ are $\phi_{n}$ and $\phi_{n+1}$ respectively from the sequence $\{\phi_{n}:n \in \omega\}$ of Lemma \ref{Lemma}, for fixed $n \in \omega$, and in such a case it of course follows from Lemma \ref{Lemma} that the sentences in question would be true in the standard model. However, for the time being we assume no such constraint on $\phi$ and $\phi'$, and we simply consider an arbitrary quadruple of the form given, noting that it is polynomial-time-computable whether or not a given string is such a quadruple or not. However we note that in the event that $\phi=\phi_{n}$ for some $n$ there would be a unique candidate for such an $n$ given the length of $\phi$. We shall assume that $n$ is chosen in such a way.

\bigskip

\noindent In what follows, we shall be speaking of a formula $\theta_n$ in the first-order language of arithmetic with exactly two free variables which can be proved in a certain extension of $\textsf{PRA}$ to define a total general recursive function. We shall need to have some kind of procedure for passing from $\theta_n$ to the index of a Turing machine which can be proved in the same extension of $\textsf{PRA}$ to compute the total general recursive function defined by $\theta_n$. Such a procedure can be extracted from the proof of computability of recursive functions in Section 68, Chapter XIII of \cite{Kleene1962}.

\begin{Proposition} \label{Step1} Assume that $\langle \phi, \phi', P, P' \rangle$ is a quadruple as described above and assume that $m, n$ are as above. Let us suppose that $q(n)$ is a fixed polynomial function of $n$ with non-negative integer coefficients. Let us define $S_n$ to be the set of $\Sigma^0_1$-formulas of length at most $q(n)$, with exactly two free variables, which define $\textsf{PRA}$+$\phi'$-provably total general recursive functions. Then, provided the initial polynomial function $q(n)$ is chosen appropriately, there will also always exist at least one formula $\theta_n \in S_n$, obtainable as a polynomial-time-computable function of our quadruple $\langle \phi, \phi', P, P' \rangle$, satisfying all of the following. In the case where $\phi$ and $\phi'$ are true in the standard model, as is the case when $\phi=\phi_n$ and $\phi'=\phi_{n+1}$ where $\phi_n$ and $\phi_{n+1}$ are taken from the sequence constructed in the previous lemma, then $f_n$ will be the total general recursive function corresponding to $\theta_n$. We now state what requirements $\theta_n$ must satisfy.

\bigskip

\noindent (1) It is $\textsf{PRA}+\phi'$-provable that the following holds for all $k$: there is a $\textsf{PRA}+\phi$-proof that there exists an $x$ such that $\theta_n(k,x)$. \newline
(2) The following is $\textsf{PRA}+\phi'$-provable. Denote by $r_n(k)$ the length of the shortest proof from $\textsf{PRA}+\phi$ of the existence of an $x$ such that $\theta_n(k,x)$. For all sufficiently large $k$, the actual halting time of the Turing machine $T_{n}$ for computing the function defined by $\theta_n$, with input $k$, lies between $(r_n(k))^k$ and $(r_n(k))^{3k}$. \newline

\bigskip

\noindent Furthermore, ``polynomial-time computable" in the fourth sentence of the first paragraph of the statement of our Proposition can be replaced by ``computable in $O(g_{k}(n))$ for some $k$", and simultaneously, similarly with other occurrences of ``polynomial", as for example in the second sentence of the first paragraph, where $\langle g_k :k \in \omega \rangle$ is any sequence of functions having the following set of properties: (1) The functions in the sequence should be closed under functional composition (2) they should be dominated growth-rate-wise by the exponential function (this requirement is necessary if we wish to ensure that the sequence of functions $\langle f_n :n \geq N \rangle$ corresponding in the obvious way to the subsequence $\langle \phi_{n} : n \geq N \rangle$ of the sequence $\langle \phi_{n} : n \in \omega \rangle$ of Lemma \ref{Lemma} -- where $N$ is taken as the least $N$ such that $\textsf{PRA}+\phi_{N}$ proves of some Turing machine that it witnesses the polynomial-time-solvability of $\textsf{PROVABILITY}$ with exponent $m$ -- is cofinal growth-rate wise in the set of all total general recursive functions) (3) we require that $n \mapsto \Sigma_{m=1}^{n} \Sigma_{j=1}^{k} a_j g_j(m)$, for arbitrary non-negative integer constants $a_j$, be dominated growth-rate-wise by $g_{k'}$ for some sufficiently large $k'$.

\end{Proposition}

\noindent The final paragraph of the statement of the above Proposition indicates the generality of our argument, notably indicating that our argument will imply as a final conclusion not only $P \neq NP$ but also that there are many sub-exponential functions $f(n)$ such that an $NP$-complete problem cannot be solved with run-time $O(f(n))$. However our proof will be consistent with the known fact \cite{Hertli2018} that the $NP$-complete problem $\textsf{3SAT}$ is solvable in $O((1.31)^n)$ (an important test for any attempted proof of $P \neq NP$, as was noted by Scott Aaronson).

\bigskip

\noindent We will also need to discuss why our construction does not work when relativized to an oracle, that is, it does not give rise to a proof of the result $P^{A} \neq NP^{A}$ for arbitrary oracles $A$, which is known to be false. We will return to this point briefly later at the appropriate time.

\bigskip

\noindent We have now clarified, then, the proposition that must be proved in order to complete Step 1 of the three steps that are required for our final conclusion, in addition to the proof of Lemma \ref{Lemma}. Let us proceed to the description of the next two steps.

\bigskip

\noindent \textbf{Description of Step 2.} Suppose that Step 1 has been completed, and let us suppose that thereby in the context of assuming $P=NP$ a sequence $\langle f_n :n \geq N \rangle$ of total general recursive functions is obtained (that is to say, each $f_n$ is the total general recursive function corresponding to the $\theta_n$ of the previous Proposition whose totality is a consequence of $\textsf{PRA}+\phi_{n+1}$ where $\phi_{n+1}$ is as in the first lemma, for all $n$ greater than or equal to the least $N$ such that $\textsf{PRA}+\phi_{N}$ proves of some Turing machine that it witnesses the polynomial-time-solvability of $\textsf{PROVABILITY}$ with exponent $m$). For each function $f_n$ in the sequence we can construct the ``Tower of Hanoi" sentence $\psi_n$, which asserts ``Either the Turing machine for $f_n(n)$ halts and I'm refutable from $\textsf{PRA}+\phi_{n}$ in at most $f_n(n)$ symbols, or the Turing machine for $f_n(n)$ does not halt and I'm refutable from $\textsf{PRA}+\phi_{n+1}$". We assume that $n$ is represented by a closed term of length polynomial in $\mathrm{log}_{2} \hspace{0.1 cm} n$, and that ``the Turing machine for $f_n(n)$ halts" is represented by ``$\exists r \theta_n(n,r)$". This construction can be achieved with the diagonal lemma, with the sentence $\psi_n$ being a polynomial-time computable function of the original quadruple $\langle \phi_{n}, \phi_{n+1}, P, P' \rangle$ (including in the case where $\phi_{n}$ and $\phi_{n+1}$ do not come from the sequence $\{\phi_n : n \in \omega \}$ of Lemma \ref{Lemma}, in which cases as noted we still use their length in symbols to come up with a unique value for $n$), and we call it the ``Tower of Hanoi" sentence, for reasons discussed below.

\bigskip

\noindent This construction was inspired by a similar construction of Hugh Woodin's in an essay \cite{Woodin1998} of his titled ``The Tower of Hanoi". Calling the construction by this name is motivated by the fact that Hugh Woodin used a similarly constructed sentence in that work to give an instance of a polynomial-time checkable property of a hypothetical string of bits of length $10^{24}$, such that there would be no feasible way to demonstrate that the string of length at most $10^{24}$ with the required property does not exist from ``anything in our mathematical experience", but also such that if the string were actually observed to exist then we would be able to infer that a certain Turing machine involving 2011 iterations of the exponential operator does not in fact halt, thereby having located in inconsistency in $\textsf{EFA}$. Naturally our mathematical intuitions lead us to believe that the string does not in fact exist, but this is not feasibly confirmable through ``direct mathematical experience" but only by intuition. The sentence would therefore reflect the ``Tower of Hanoi" legend in that sense, since the ``Tower of Hanoi" legend holds that if the Tower of Hanoi game with 64 disks were ever completed then the gods would step in and destroy the universe, analogous to the ``destruction of the mathematical universe" that would ensue from actually locating the bit-string. So for this reason we call it a ``Tower of Hanoi" sentence but will be using this construction for very different purposes.

\bigskip

\noindent Given that the ``Tower of Hanoi" sentences are constructed as above, we construct a partial general recursive function $g$ in the following way. We search through every possible quadruple $\langle \phi, \phi', P, P' \rangle$ of the form that we have been considering, and given such a quadruple we may construct a ``Tower of Hanoi" sentence $\psi$ corresponding to this quadruple which may be obtained as a polynomial-time function of the quadruple.

\bigskip

\noindent We now define a partial general recursive function $g$ as follows. For every quadruple as specified above with length bounded above by some polynomial function of the argument $k$ to the function $g$, we search either for a refutation from $\textsf{PRA}+\phi$ of the Tower of Hanoi sentence $\psi$, or else a proof in $\textsf{PRA}+\phi+\neg\psi$ of $\Sigma^0_r$-unsoundness, for some $r$, of $\textsf{PRA}+\phi'+\neg\psi$. We claim that this search always terminates for every tuple with the desired (polynomial-time-checkable) properties, no matter what. But, pending a demonstration of that, then of course we simply stipulate that if the search fails to terminate for some tuple corresponding to a given argument $k$ then that simply counts as a value of $k$ for which the function $g$ is undefined. Then, having performed such a search and brought it to termination for every tuple with length bounded above by the appropriate polynomial bound, we take the maximum length of all proofs found as the outcome of the search over all the tuples, and that is the value of $g(k)$ for this $k$. In this way we have clearly defined a partial general recursive function, computer code for which it is totally feasible to construct.

\begin{Proposition} \label{Step2} The function $g$ as defined above is in fact provably total in $\Sigma^0_2$-induction arithmetic. \end{Proposition}

\noindent The proof of the above proposition is what must be achieved for the completion of Step 2.

\bigskip

\noindent \textbf{Description of Step 3.} Suppose that we have completed Steps 1 and 2, and are assuming $P=NP$ for a contradiction, and suppose that we are also assuming $\Sigma^0_3$-induction so that we can assume that predicates defined with reference to our sequence of sentences $\langle \phi_n : n \in \omega \rangle$ constructed in Lemma \ref{Lemma} are such that if they are inductive then their extension is the whole set of natural numbers ($\Sigma^0_3$-induction appearing to be the amount of induction that is required for this on careful inspection of the proof).

\begin{Proposition} \label{Step3} Given the above assumptions and the previous step, the function $g$ is in fact a total general recursive function growing faster than any $f_n$ in the sequence $\langle f_n : n \geq N \rangle$ discussed at the beginning of Step 2, but this sequence of functions is cofinal growth-rate-wise in the set of total general recursive functions (as we will make it our burden to show in the completion of Step 1), and that's a contradiction. Thus we've derived a contradiction from $P=NP$ in $\Sigma^0_3$-induction arithmetic, therefore we can now conclude that $P \neq NP$ is provable in $\Sigma^0_3$-induction arithmetic and in particular in $\textsf{PA}$, and this completes the proof that $\textsf{PA}$ proves $P \neq NP$. \end{Proposition}

\noindent This then, is the general sketch of how the proof goes. So to complete the proof, we must prove Lemma \ref{Lemma}, Proposition \ref{Step1} (including the demonstration that $\langle f_n:n \geq N \rangle$ is cofinal growth-rate-wise in the set of all total general recursive functions), Proposition \ref{Step2}, and Proposition \ref{Step3}, and then the final conclusion will be that $\textsf{PA}$ proves $P \neq NP$. The presentation of the proofs of these claims will be the task of the fifth section.

\bigskip

\noindent We will also, in the course of proving these claims, need to briefly note towards the end of the proof of Proposition \ref{Step1} the reasons why our argument is not a ``relativizing" argument; it does not succeed if relativized to an oracle. Our task of outlining the overall shape of the argument is now complete.

\section{Proof of the main theorem}

\noindent We present the proof of Lemma \ref{Lemma}.

\begin{lemma} \label{Lemma_1} Assume the axioms of $\mathsf{PA}$. On those assumptions, it is provable that there exists an arithmetically definable sequence of $\Pi^{0}_{2}$-sentences $\{\phi_{n}:n \in \omega\}$ in the language of $\textsf{PRA}$ as defined above (we allow that they may be finite conjunctions of sentences of complexity either $\Pi^0_1$ or $\Pi^0_2$, and thus provably equivalent in $\textsf{PRA}$ to some $\Pi^0_2$-sentence) with the property that

\bigskip

\noindent (1) No $q$-marks of length greater than $n$ occur in $\phi_{n}$ (and we start with $n=0$). \newline
\noindent (2) There exists a polynomial function with non-negative integer coefficients such that the length in symbols of $\phi_{n}$ is always bounded above by the value of this polynomial function at $n$, and is also strictly greater than the value of this polynomial function at $n-1$ if $n>0$; \newline
(3) $\mathsf{PRA}$+$\{\phi_{n}:n\in\omega\}$ is $\Pi^{0}_{2}$-sound and $\Pi^{0}_{1}$-complete \newline
(4) $\mathsf{PRA}$+$\phi_{n+1}$ always proves 1-consistency of $\mathsf{PRA}$+$\phi_{n}$, in a form where the G\"odel number of $\phi_{n}$ is denoted by a constant symbol with $q$-marks, and the length of the shortest proof as a function of $n$ is bounded above by a polynomial function with non-negative integer coefficients; \newline
(5) The provably total general recursive functions of $\mathsf{PRA}$+$\{\phi_{n}:n \in\omega\}$ have growth rates cofinal in the set of growth rates of all the total general recursive functions. \end{lemma}

\begin{proof} Our background metatheory is $\mathsf{PA}$. Choose a positive integer $C$ such that, given a $\Pi^{0}_{2}$-sentence $\phi$ of length $N$ symbols with no $q$-marks of length greater than or equal to $n$, a $\Pi^{0}_{2}$-sentence $\psi$ of length at most $N+2n+C$ symbols exists which (together with $\mathsf{PRA}$) implies 1-consistency of $\mathsf{PRA}$+$\phi$. The construction of the desired sequence $\{\phi_{n}:n \in \omega\}$ is as follows.

\bigskip

\noindent Define $p(n)$ to be a polynomial with constant term 100, degree at least 2, and such that $p(n)-p(n-1)\geq 3n+C$ for all $n>0$. Select the sentence $\phi_{n}$ as follows, starting from $n=0$ and proceeding by induction on $n$, so assuming at each step that all previous sentences have already been constructed. If we are in the case $n>0$, then consider the set $S$ of all true $\Pi^{0}_{2}$-sentences with more than $p(n-1)$ symbols and at most $p(n-1)+2n+C$ symbols and no $q$-marks of length greater than $n$ which (together with $\mathsf{PRA}$) imply 1-consistency of $\mathsf{PRA}$+$\phi_{n-1}$, otherwise just consider the set $S$ of all true $\Pi^{0}_{2}$-sentences with at most 50 symbols and no $q$-marks; in either case this set is non-empty and finite. Each sentence in $S$ is true in the standard model, and so is associated with a particular general recursive function, of which the sentence asserts totality. The total general recursive functions are partially ordered by comparison of growth rates, so in this way $S$ may be viewed as a partially ordered set.

\bigskip

\noindent Our sentence $\phi_{n}$ will be a conjunction of two sentences, the first one a $\Pi^{0}_{2}$-sentence chosen as follows. We choose our $\Pi^{0}_{2}$-sentence by selecting a maximal element of the partially ordered set $S$ at each stage, also strictly dominating growth-rate-wise, in the case $n>0$, each total general recursive function whose totality is witnessed by a true $\Pi^0_2$-sentence of length at most $q(n)$ (and with $q$-marks of arbitrary length allowed), where $q(n)$ is an appropriately chosen linear polynomial with non-negative coefficients and positive leading coefficient. With the polynomial $q(n)$ fixed, this is indeed possible for an appropriate choice of the polynomial $p(n)$.

\bigskip

\noindent Our second conjunct is chosen (where possible) from the set of true $\Pi^{0}_{1}$-sentences with no $q$-marks taken in lexicographic order, and of length bounded so as to make the length of the conjunction at most $p(n)$. For appropriate choices of all the polynomial bounds involved, this gives rise to a sequence with the desired set of properties. The sequence constructed in this way is clearly arithmetically definable.

\end{proof}

\bigskip

\noindent We present the proof of Proposition \ref{Step1}.

\begin{Proposition} \label{Step1_1} Assume that $\langle \phi, \phi', P, P' \rangle$ is a quadruple as described above and assume that $m, n$ are as above. Let us suppose that $q(n)$ is a fixed polynomial function of $n$ with non-negative integer coefficients. Let us define $S_n$ to be the set of $\Sigma^0_1$-formulas of length at most $q(n)$, with exactly two free variables, which define $\textsf{PRA}$+$\phi'$-provably total general recursive functions. Then, provided the initial polynomial function $q(n)$ is chosen appropriately, there will also always exist at least one formula $\theta_n \in S_n$, obtainable as a polynomial-time-computable function of our quadruple $\langle \phi, \phi', P, P' \rangle$, satisfying all of the following. In the case where $\phi$ and $\phi'$ are true in the standard model, as is the case when $\phi=\phi_n$ and $\phi'=\phi_{n+1}$ where $\phi_n$ and $\phi_{n+1}$ are taken from the sequence constructed in the previous lemma, then $f_n$ will be the total general recursive function corresponding to $\theta_n$. We now state what requirements $\theta_n$ must satisfy.

\bigskip

\noindent (1) It is $\textsf{PRA}+\phi'$-provable that the following holds for all $k$: there is a $\textsf{PRA}+\phi$-proof that there exists an $x$ such that $\theta_n(k,x)$. \newline
(2) The following is $\textsf{PRA}+\phi'$-provable. Denote by $r_n(k)$ the length of the shortest proof from $\textsf{PRA}+\phi$ of the existence of an $x$ such that $\theta_n(k,x)$. For all sufficiently large $k$, the actual halting time of the Turing machine $T_{n}$ for computing the function defined by $\theta_n$, with input $k$, lies between $(r_n(k))^k$ and $(r_n(k))^{3k}$. \newline

\bigskip

\noindent Furthermore, ``polynomial-time computable" in the fourth sentence of the first paragraph of the statement of our Proposition can be replaced by ``computable in $O(g_{k}(n))$ for some $k$", and simultaneously, similarly with other occurrences of ``polynomial", as for example in the second sentence of the first paragraph, where $\langle g_k :k \in \omega \rangle$ is any sequence of functions having the following set of properties: (1) The functions in the sequence should be closed under functional composition (2) they should be dominated growth-rate-wise by the exponential function (this requirement is necessary if we wish to ensure that the sequence of functions $\langle f_n :n \geq N \rangle$ corresponding in the obvious way to the subsequence $\langle \phi_{n} : n \geq N \rangle$ of the sequence $\langle \phi_{n} : n \in \omega \rangle$ of Lemma \ref{Lemma} -- where $N$ is taken as the least $N$ such that $\textsf{PRA}+\phi_{N}$ proves of some Turing machine that it witnesses the polynomial-time-solvability of $\textsf{PROVABILITY}$ with exponent $m$ -- is cofinal growth-rate wise in the set of all total general recursive functions) (3) we require that $n \mapsto \Sigma_{m=1}^{n} \Sigma_{j=1}^{k} a_j g_j(m)$, for arbitrary non-negative integer constants $a_j$, be dominated growth-rate-wise by $g_{k'}$ for some sufficiently large $k'$.

\end{Proposition}

\begin{proof} We must describe how the formula $\theta_n$ with exactly two free variables is obtained as a polynomial-time-computable function of the quadruple $\langle \phi, \phi', P, P' \rangle$. We are only requiring that $\theta_n$ must define a $\textsf{PRA}+\phi'$-provably total general recursive function, as opposed to a total general recursive function in the standard model, so, working in $\textsf{PRA}+\phi'$ and letting $f_n$ denote the total general recursive function defined by the formula (on the assumption of $\textsf{PRA}+\phi'$), we must define an algorithm for computing $f_n(k)$ for a given $k$, such that the description of this algorithm is indeed polynomial-time-computably obtainable from the description of our quadruple and also this algorithm is such that $\textsf{PRA}+\phi'$-provably it halts for all $k$. We shall proceed to describe the algorithm. The value of $n$ will be held fixed in what follows and when we speak of ``polynomially bounded" we will be referring to bounds given by polynomial functions of $k$.

\bigskip

\noindent It does not matter what the algorithm does in the case $k=0$; we can say that it immediately halts and outputs $0$. The discussion that follows will be confined to the case $k>0$.

\bigskip

\noindent \begin{Definition} The decision problem $\textsf{PROVABILITY}$ is defined as in Lemma \ref{decisionproblem}.

\end{Definition}

\noindent Our algorithm shall involve searching through a certain bounded set of indices of Turing machines such that it is $\textsf{PRA}+\phi'$-provable that all Turing machines with indices in the set halt on empty input, and then letting $f_n(k)$ be the maximum halting time for any Turing machine with an index in the set. We shall assume that a certain ``speed-up method" is employed in the algorithm for computing $f_n(k)$, based on the fact that our quadruple includes a proof $P'$ from $\textsf{PRA}+\phi$ of the polynomial-time-computability of the decision problem $\textsf{PROVABILITY}$. To elaborate further, we shall need to specify what it is for a Turing machine to have Property A and Property B, for each $k>0$. Recall that at the start of the statement of our Proposition we made a stipulation about a particular value of $n$ which will be held fixed throughout.

\bigskip

\noindent We shall assume in what follows that $r(k), s(k)$ are two fixed polynomial functions of $k$ with non-negative integer coefficients; we shall say more about how they are to be specified later.

\begin{Definition} Let $\{\chi_r:r \in \omega\}$ be the collection of all the sentences asserting 1-consistency of $\Gamma_{r}+\phi$. The Turing machine $T$ is said to have Property A if there is a proof $p_1$ from $\textsf{PRA}+\{\chi_r:r\in\omega\}$ of length at most $r(k)$ of the existence of a proof $p_2$ from $\textsf{PRA}+\phi$ of arbitrary length of the halting of the Turing machine $T$ when given empty input, where we assume that $T$ is given by a description determined in a natural way from its index $e$, the description's length in symbols being $O(\mathrm{log}_{2} \hspace{0.1 cm} e)$. \end{Definition}

\begin{Definition} Suppose that the Turing machine $T$ has Property A. Then there is a proof $p_1$ from $\textsf{PRA}+\{\chi_r:r\in\omega\}$ as above. Suppose that in addition to this proof, there is a proof $p_3$ from $\textsf{PRA}+\{\chi_r:r\in\omega\}$ of length at most $r(k)$, of the lower and upper bounds stipulated in requirement (2) of the statement of our Proposition, with the Turing machine $T$ with empty input substituted for the Turing machine $T_{n}$ with input $k$ referred to there, and with a pseudo-term for the smallest possible length of the proof $p_2$ substituted for $r_n(k)$, and with the exponent $2k$ substituted for the exponent $3k$. Then the Turing machine $T$ is said to have Property B. (Note that every proof in $\textsf{PRA}+\{\chi_r:r\in\omega\}$ has a corresponding proof of the same conclusion in $\textsf{PRA}+\phi'$ with only a small increase in length.) \end{Definition}

\bigskip

\noindent Recalling that $\textsf{PRA}+\phi'$ is our background theory here, we must establish the claim that when the polynomial function $r(k)$ is suitably chosen, the bounded set of all indices of Turing machines with Property A and Property B for a fixed $k>0$ can be seen to be non-empty (and the universal closure of this statement with the restriction $k>0$ is provable in $\textsf{PRA}+\phi'$). This can be easily inferred from the evident claim that for all natural numbers $r, s$ there is a Turing machine $T$ such that $\textsf{PRA}+\phi$ has shortest proof of halting of $T$ with length $C+\mathrm{log} \hspace{0.1 cm} 2 r+\mathrm{log} \hspace{0.1 cm} 2 s$ and halting time $r^s+C'$ for some constants $C, C'$.

\bigskip

\noindent Now let us continue to specify our algorithm. Our algorithm shall search through the bounded non-empty set of all indices of Turing machines with Property A and Property B, and select the Turing machine in the set with largest possible halting time, as computed in the following way. We ``search through the space of Turing machines" with both Property A and Property B and run them all in sequence, where we assume that we are using the polynomial-time solver for the decision $\textsf{PROVABILITY}$ which $\textsf{PRA}+\phi'$-provably gives correct answers as per the proof $P'$ to speed up the search process. We may in fact have a failure to halt for one of the Turing machines in the set if $\textsf{PRA}+\phi'$ is inconsistent, but in that case we just search for a proof from $\textsf{PRA}+\phi'$ of a contradiction and if one is found simply halt the process and return the Turing machine in the set with smallest possible index. It is easy to see that this Turing machine will have Property $\mathrm{A}^{*}$, defined as with property A but with $s(k)$ substituted for $r(k)$, provided $r(k)$ and $s(k)$ are chosen appropriately. This Turing machine will also have Property $\mathrm{B}^{*}$, similarly defined as with property B but with $s(k)$ substituted for $r(k)$, and with the exponent $3k$ substituted for the exponent $2k$, provided that $k$ is a lot larger than $m$, the exponent for the polynomial time in which the decision problem $\textsf{PROVABILITY}$ is solvable with exponent $m$ via the proof $P'$. To see this, note that this claim as applied to the version of Property $\mathrm{B}^{*}$ involving the upper bound on halting time alone, when $s(k)$ is suitably chosen, is clear, and given suitable choices for $r(k)$ and $s(k)$ a Turing machine will exist with its index in the set of indices of Turing machines in question, with its run-time ``padded out" if need be so that the required lower bound on halting time will hold.

\bigskip

\noindent Note further that if we re-run this entire argument with ``polynomial-time-computable" replaced by ``computable in $O(g_k(n))$ steps for some $k$" as in the final paragraph of the statement of the Proposition, with one of the $g_k$ dominating the exponential function growth-rate-wise, then $f_n(k)$ may be computable by some fixed Turing machine with input $k$ which always halts given arbitrary input provably in $\textsf{PRA}+\phi$, so that the sequence of total general recursive functions $\langle f_n:n \geq N \rangle$ constructed from the sequence $\{\phi_{n}:n \in \omega\}$ of Lemma \ref{Lemma} may fail to be co-final growth-rate-wise in the set of all total general recursive functions. On the other hand in the situation where all the $g_k$ are slower growth-rate-wise than any exponential function this cannot occur and the sequence $\langle f_n:n \geq N \rangle$ will be cofinal growth-rate-wise in the set of all total general recursive functions.

\bigskip

\noindent If we try to relativize to an oracle $A$ while assuming $P^{A}=NP^{A}$, then the attempt to show that Properties $\mathrm{A}^{*}$ and $\mathrm{B}^{*}$ will always hold for the Turing machine with optimal halting time will fail. Because with regard to the proofs considered in the statement of Properties A and B, and Properties $\mathrm{A}^{*}$ and $\mathrm{B}^{*}$, we will basically be speaking of proofs where we're allowed to use any sentence we want from an axiom schema asserting membership or failure of membership in the set $A$, and when we optimize running time over the entire set of Turing machines, we may not be able to polynomially bound how many such axioms we need to use in order to construct our proofs witnessing Properties $\mathrm{A}^{*}$ and $\mathrm{B}^{*}$, from the proofs witnessing Properties A and B for each Turing machine in the set. So this shows why our argument is not a ``relativizing" one.

\bigskip

\noindent The claims made by the Proposition are now all clear.

\end{proof}

\noindent We present the proof of Proposition \ref{Step2}.

\begin{Proposition} \label{Step2_1} Given that the ``Tower of Hanoi" sentences are constructed as in Section 3, we construct a partial general recursive function $g$ in the following way. We search through every possible quadruple $\langle \phi, \phi', P, P' \rangle$ of the form that we have been considering, and given such a quadruple we may construct a ``Tower of Hanoi" sentence $\psi$ corresponding to this quadruple, as described in Section 3, which may be obtained as a polynomial-time function of the quadruple. For every quadruple as specified above with length bounded above by some polynomial function of the argument $k$ to the function $g$, we search either for a refutation from $\textsf{PRA}+\phi$ of the Tower of Hanoi sentence $\psi$, or else a proof in $\textsf{PRA}+\phi+\neg\psi$ of $\Sigma^0_r$-unsoundness, for some $r$, of $\textsf{PRA}+\phi'+\neg\psi$. If this search fails to terminate for a tuple corresponding to a given argument $k$ then we stipulate that $g$ is undefined for that value of $k$. However, we claim that this search always terminates (provably in $\Sigma^0_2$-induction arithmetic) for every tuple with the desired (polynomial-time-checkable) properties, and having performed such a search and brought it to termination for every tuple with length bounded above by the appropriate polynomial bound, we take the maximum length of all proofs found as the outcome of the search over all the tuples, and that is the value of $g(k)$ for this $k$. Our claim, then, is that the function $g$ so defined is provably total in $\Sigma^0_2$-induction arithmetic. \end{Proposition}

\begin{proof} Suppose that the search never terminates. Then $T:=\mathsf{PRA}+\phi+\neg\psi+\{$``$\mathsf{PRA}+\phi'+\neg\psi$ is $\Sigma^0_r$-sound":$r\in\omega\}$ is consistent and $\mathsf{PRA}+\phi+\psi$ is consistent. We are going to be doing model-theoretic reasoning in $\textsf{WKL}_0$ (see \cite{Simpson1999} for a definition of this theory and a discussion of its basic properties), which is adequate for the basic results of model theory including the proof of the completeness theorem and which is $\Pi^0_2$-conservative over $\textsf{PRA}$, as discussed, again, in \cite{Simpson1999}.

\bigskip

\noindent Let $M$ be a model of $T$. $M$ is a model for (schematic) arithmetical soundness of $\mathsf{PRA}+\phi'+\neg\psi$, and $\mathsf{PRA}+\phi'$ predicts (with a proof of standard length) that $\mathsf{PRA}+\phi$ will eventually refute $\psi$, as can be seen from the way that $\psi$ is constructed. The reason for this is that $\mathsf{PRA}+\phi'$ predicts that the Turing machine referred to in the sentence $\psi$ does indeed halt. And, if the Turing machine halts, then clearly one way or another a refutation of $\psi$ from $\mathsf{PRA}+\phi$ will eventually be found, and so $\mathsf{PRA}+\phi'$ makes this prediction. Thus in the model $M$ is a proof $p$ in $\mathsf{PRA}+\phi$ of $\neg\psi$.

\bigskip

\noindent We are going to be speaking of ``standard" and ``non-standard" elements of a model, so it will be desirable to clarify what is meant by ``standard". We will eventually want to claim that all of our reasoning can be formalised in $\mathsf{WKL_0}$ which is $\Pi^0_2$-conservative over $\mathsf{PRA}$, but currently we are assuming for a contradiction the consistency of $T$, and then given that assumption $\mathsf{WKL_0}$ implies the existence of a model $M$ of $T$. There is a cut of such a model -- a $\Sigma^0_1$-definable class relative to $M$, where we should note that in $\textsf{WKL}_0$ we do not have a set existence axiom for arbitrary $\Sigma^0_1$-definable classes, but we can still talk about them meta-linguistically in the usual way -- such that this cut is the least cut that models $\textsf{PRA}$+``$\textsf{PRA}$ is 1-consistent". We can indeed assume that such a cut does exist as a $\Sigma^0_1$-definable class, given that we are assuming that $T$ is consistent, and given how strong $T$ is.

\bigskip

\noindent Elements of $M$ will be said to be standard if they belong to this cut. Thus it will be okay in what follows to assume that the standard fragment of $M$ is closed under Ackermann's function. So, with that understanding of what ``standard" means, could $p$ be a standard element of the model $M$? If that were the case, then with the further assumption that Ackermann's function is total (in the meta-theory, as an extra assumption added to $\textsf{WKL}_0$) we could conclude that the search of which we spoke in the statement of our Proposition terminates. Our final conclusion will be that $\mathsf{WKL_0}$, which is $\Pi^0_2$-conservative over $\mathsf{PRA}$, proves either the termination of the search or the termination of the search on the assumption that Ackermann's function is total. That will be enough for the function $g$ to be provably total from $\Sigma^0_2$-induction arithmetic, and so will establish our claim. Given this consideration, we will be able to assume without loss of generality in what follows that $p$ is a non-standard element of $M$ (since we already see that the desired conclusion is easy to obtain on the assumption that $p$ is standard).

\bigskip

\noindent Denote by $T'$ the theory $\mathsf{PRA}$+$\phi$+$\{$``$\Gamma_{r}+\phi$ is 1-consistent"$:r \in \omega\}$. The least cut $M'$ of $M$ satisfying $T'$ satisfies $\mathsf{PRA}$+$\phi+\neg\psi$ since $\neg\psi$ is $\Pi^0_1$, and it can be seen -- recalling that the functions $f_{n}$ were constructed so that a proof of halting of a Turing machine for $f_{n}(k)$ could always be found in $T'$ in a number of steps bounded above by a polynomial function of $k$ -- that $p$ may be chosen so that $p \in M'$, and we can assume without loss of generality that $p$ is non-standard by the foregoing. Denote by $A$ the (standard) Turing machine which plays in $\psi$ the part that was played by the Turing machine for $f_n(n)$ in $\psi_n$. That is, we are speaking of the Turing machine which computes the appropriate analogue of $f_n(n)$ when given empty input.

\bigskip

\noindent Since $M'$ is a cut of a model of $T$, and $p$ is a non-standard element of $M'$, the halting time of $A$ in $M'$ is non-standard. Then the shortest proof of halting of $A$ from $\mathsf{PRA}$+$\phi$ has non-standard length, and when raised to the power of a sufficiently large standard natural number, it exceeds the halting time for $A$ (as we see when we recall the way in which $A$ was constructed in the proof of the previous Proposition). One can then also conclude that for each standard natural number $k$, the length of the shortest proof from $\mathsf{PRA}+\phi$ of ``$\mathsf{PRA}+\phi$ proves $\mathsf{PRA}+\phi$ proves $\mathsf{PRA}+\phi$ proves... $A$ halts", going $k$ levels deep, is also non-standard. Let us discuss how to argue for this point.

\bigskip

\noindent We note that the shortest proof in $\mathsf{PRA}+\phi$ of ``$A$ halts" is ``close to" the actual halting time of $A$ in the sense that if $a$ is the length of the shortest proof and $h$ is the halting time then $a^{3r} \geq h$ for a standard $r$. It follows from the construction of $A$ that the length of the shortest proof in $\mathsf{PRA}+\phi$ of ``$\mathsf{PRA}+\phi$ proves $A$ halts" must also be non-standard (since the length of this proof exceeds any element of the model that equals the halting time of any standard Turing machine with a standard proof of halting in $\mathsf{PRA}+\phi$). In addition, if $a'$ is the length of the shortest proof of ``$\mathsf{PRA}+\phi$ proves $A$ halts" then again $(a')^{3r} \geq h$ for a sufficiently large standard $r$. Proceeding by induction, if we let $T'$ be the Turing machine which searches for the shortest proof in PRA+$\phi$ of ``$\mathsf{PRA}+\phi$ proves $\mathsf{PRA}+\phi$ proves... $A$ halts", for any finite number of levels deep, the length of the shortest proof of halting of $T'$ is non-standard. We want to argue that given the way the Turing machine $A$ is constructed, this gives rise to a contradiction. The nature of the argument for this point is basically proof-theoretic.

\bigskip

\noindent One can introduce function symbols for every primitive recursive function and also for the following total general recursive functions (relative to our model $M'$ which is a model for $T'$ with no proper cuts which model $T'$): the function $h_0$ whose totality is asserted by $\phi$ and the each of the functions $h_r$ for $r>0$ whose totality is asserted by the sentence ``$\Gamma_{r}$+$\phi$ is 1-consistent". Then one can construct a ``quantifier-free" system in a language with these function symbols, with open formulas as axioms, with the first-order theory $T'$ being $\Pi^0_2$-conservative over this \newline quantifier-free system in an appropriate sense (where here we are dealing entirely with proofs of standard length).

\bigskip

\noindent We can now consider that there will be a closed term of standard length in this language which denotes the halting time of $A$ in every model of $T'$. If we recall the construction of $A$, we see that such a term of standard length can be obtained computationally by searching through the space of proofs in $T'$ of standard length bounded above by a certain fixed standard upper bound. Furthermore, it can be proved in $\mathsf{PRA}$+``$\mathsf{PRA}$ is 1-consistent" -- which we are allowed to use here given the previous discussion of how we are using ``standard" in such a way that it is okay to assume that the standard fragment of any model of $T$ or $T'$ models $\mathsf{PRA}$+``$\mathsf{PRA}$ is 1-consistent" -- that every such term is bounded above (in every model of $T'$) by another term of standard length, where nested occurrences of $h_r$ in the first term are replaced by a single occurrence of $h_s$ for a larger $s$, such that for every occurrence of $h_s$ with $s>0$ in the term, the argument to $h_s$ is a term with the property that it denotes a standard natural number in every model of $T'$. Having noted this, we then see that this term of standard length can in turn be proved in $\mathsf{PRA}$+``$\mathsf{PRA}$ is 1-consistent" to be bounded above (in every model of $T'$) by another term of standard length in which only the primitive recursive functions and $h_0$ occur. From the existence of such a term of standard length one can now obtain the result that there exists some standard natural number $k$ such that $\mathsf{PRA}$+$\phi$ proves $\mathsf{PRA}$+$\phi$ proves... $A$ halts, where we only go $k$ levels deep. But this contradicts the observation previously made.

\bigskip

\noindent We are now at the point where we have derived a contradiction (in $\mathsf{WKL_0}$+``Ackermann's function is total") from supposing that the theory $T$ has a model and also that $\mathsf{PRA}$+$\phi+\psi$ is consistent, but this supposition was seen to be a consequence of the assumption that the search never terminates. Our conclusion is that the search does provably terminate in $\Sigma^0_2$-induction arithmetic. Hence the function $g$ is provably total in $\Sigma^0_2$-induction arithmetic as claimed. This completes the proof of Proposition \ref{Step2_1}. \end{proof}

\begin{Proposition} \label{Step3_1} Given the above assumptions and the previous step, the function $g$ is in fact a total general recursive function growing faster than any $f_n$ in the sequence $\langle f_n : n \geq N \rangle$ discussed at the completion of Step 1, but this sequence of functions is cofinal growth-rate-wise in the set of total general recursive functions, and that's a contradiction. Thus we've derived a contradiction from $P=NP$ in $\Sigma^0_3$-induction arithmetic, therefore we can now conclude that $P \neq NP$ is provable in $\Sigma^0_3$-induction arithmetic and in particular in $\textsf{PA}$, and this completes the proof that $\textsf{PA}$ proves $P \neq NP$. \end{Proposition}

\begin{proof} It is clear from the proof of Proposition \ref{Step1_1} that the sequence of functions $\langle f_n : n \geq N \rangle$ is cofinal growth-rate-wise in the set of total general recursive functions, and it is clear from the construction of the ``Tower of Hanoi" sentences that the function $g$ grows faster than any of them, since the shortest refutation of $\psi_n$ has length $f_n(n)$. Thus our claim follows. \end{proof}

From all of the foregoing discussion we finally have

\begin{theorem} It is provable in $\mathsf{PA}$ that $P \neq NP$. \end{theorem}

\begin{proof} Since we have now derived a contradiction in $\textsf{PA}$ from the assumption that $P=NP$, this completes the proof in $\textsf{PA}$ that $P \neq NP$. \end{proof}

\section{Showing that the argument works in $\textsf{SEFA}$}

\noindent We have completed the proof in $\textsf{PA}$ that $P \neq NP$. We shall describe how to adapt this proof to show that $\Gamma_5$ proves $P \neq NP$, and that the total general recursive function witnessing the truth of the $\Pi^0_2$-statement $P \neq NP$ is elementary recursive. We will conclude by showing how the argument can in fact be adjusted to show that $\Gamma_4$, that is $\textsf{SEFA}$, proves $P \neq NP$.

\bigskip

\noindent We begin by observing that when we used the theories $\Gamma_r$ in our construction of $f_n$ we could have replaced them by $I\Delta_0+\Omega_r$.

\bigskip

\noindent These theories in versions that use cut-free axiomatizations of first-order logic are all provably 1-consistent in $\Gamma_3$ (and the universally quantified statement asserting that all of them are 1-consistent is provable in $\Gamma_3$). From this consideration, combined with considerations of the speed at which proofs can be found of the form searched for in the algorithm for the function $g$, as can be observed from examining the proof of Proposition \ref{Step2_1}, we can in fact see that the version of the function $g$ defined this way would be elementary recursive.

\bigskip

\noindent We could also have replaced every occurrence of $\textsf{PRA}$ with $\textsf{EFA}$ (in a suitably extended language with all the appropriate constant symbols, that is), while still having $g$ come out elementary recursive. There is a weak version of $\textsf{WKL}_0$ called $\textsf{WKL}^{*}_0$ which is $\Pi^0_2$-conservative over $\textsf{EFA}$, which is presented in Chapter X of \cite{Simpson1999}, and which is sufficiently strong to produce all the model-theoretic results that we needed in the previous section. We could have have used $\textsf{WKL}^{*}_0+\Gamma_5$ which is $\Pi^0_2$-conservative over $\Gamma_5$ to construct a model of $\textsf{EFA}$ which would serve for the construction of the sequence $\langle \phi_n : n \in \omega \rangle$, (these sentences being required to be true in this model, which would be necessarily non-standard), and the principle of induction for sets defined from this sequence using set existence axioms allowable in $\textsf{WKL}^{*}_0$ would be permitted in $\textsf{WKL}^{*}_0+\Gamma_5$, which again is $\Pi^0_2$-conservative over $\Gamma_5$. We require adding $\Gamma_5$ to $\textsf{WKL}^{*}_0$ in order to be able to prove 1-consistency of $\textsf{EFA}$. Thus we end up with the conclusion that the $\Pi^0_2$-sentence $P \neq NP$ is provable in $\Gamma_5$ and the total general recursive function witnessing its truth is elementary recursive, as claimed.

\bigskip

\noindent We can in fact pull the argument all the way down into $\textsf{SEFA}$. To achieve this, substitute $\textsf{EFA}$ everywhere with $I\Delta_0+\{\Omega_r : r \in \omega\}$, so that the 1-consistency claim for this theory is provable in $\textsf{SEFA}$. In this way we can dispense with the assumption of $\Gamma_5$ and replace it with $\Gamma_4$. Thus our final conclusion is that $\textsf{SEFA}$ proves $P \neq NP$.

\end{document}